\documentclass[11pt,oneside]{article}
\usepackage{amsfonts}
\usepackage{amsmath}
\usepackage{amsthm}  
\usepackage[mathscr]{euscript}
\usepackage{mathrsfs}
\usepackage{graphicx}
\usepackage{subfigure}
\usepackage{varwidth}
\usepackage{color}\usepackage{enumitem}
\usepackage{url}
\usepackage{authblk}

\newcommand{\lbl}[1]{\label{#1}}
\newtheorem{theo}{Theorem}[section]

\newtheorem{lem}{Lemma}[section]

\newtheorem{defi}{Definition}[section]
\newcommand{\be}{\begin{equation}}
\newcommand{\ee}{\end{equation}}
\newcommand\bes{\begin{eqnarray}} \newcommand\ees{\end{eqnarray}}
\newcommand{\bess}{\begin{eqnarray*}}
\newcommand{\eess}{\end{eqnarray*}}

\newcommand\ep{\varepsilon}

\newcommand\kk{\left}
\newcommand\rr{\right}
\newcommand\dd{\displaystyle}
\newcommand\vp{\varphi}

\newcommand\qq{\eqref}
\newcommand\lm{\lambda}

\newcommand\yy{\infty}

\newcommand\R{\mathbb{R}}
\newcommand\up{\overline}
\newcommand\ud{\underline}
\newcommand\w{\omega}

\newcommand\oo{\Omega}\newcommand\ooo{\Sigma}
\newcommand\rw{\rho\omega}
\newcommand\boo{\bar\Omega}
\title{The nonlocal dispersal equation with seasonal succession
\footnote{The work of the first author was supported by NSFC Grant 12201457, and the work of the second author was supported by NSFC Grant 12171120.}}

\author[a]{Qianying Zhang}

\author[,\,b]{Mingxin Wang\footnote{Corresponding author. {\sl E-mail}: mxwang@hpu.edu.cn}}

\affil[a]{\small\it School of Mathematical Sciences, Tiangong University, Tianjin 300387, China}

\affil[b]{\small School of Mathematics and Information Science, Henan Polytechnic University, Jiaozuo 454000, China}

\date{} 

\begin{document}
\setlength{\baselineskip}{16pt} \pagestyle{myheadings}
\maketitle
\vspace{-1cm}

\begin{quote}
\noindent{\bf Abstract.} In this paper, we focus on the nonlocal dispersal monostable equation with seasonal succession, which can be used to describe the dynamics of species in an environment alternating between bad and good seasons. We first prove the existence and uniqueness of global positive solution, and then discuss the long time behaviors of solution. It is shown that its dynamics is completely determined by the sign of the principal eigenvalue, i.e., the time periodic problem has no positive solution and the  solution of the initial value problem tends to zero when principal eigenvalue is non-negative, while the time periodic positive solution exists uniquely and is globally asymptotically stable when principal eigenvalue is negative.


 \end{quote}

 \section{Introduction} \lbl{s1}
 \setcounter{equation}{0} {\setlength\arraycolsep{2pt}

Seasonal succession is a common natural phenomena and has an important effect on the growth and interactions of species. There have been several works in this aspect. For example, the authors of \cite{HZ12, LiZhao16} studied the ODE equation and system with seasonal succession; the authors of \cite{ZZ13, MZ16, WZZnon2022} considered the travelling waves of the diffusive competition model with seasonal succession; the authors of \cite{PZ13} and \cite{WZZjde2021} investigated the free boundary problem with seasonal succession of the diffusive logistic model and competition model, respectively.

In the above cited papers, the diffusion is the random dispersal, which is called local diffusion. Although this is a reasonable approximation of the actual dispersal in many situations, it is increasingly recognized that such an approximation is not good enough in general \cite{NKRR}. For example, movements and interactions of some organisms occur widely between non-adjacent spatial locations (such as spreadings caused by seeds or insects carried to new environment by modern ways of transportation). The evolution of nonlocal dispersal has attracted a lot of attentions for both theoretically and empirically. An extensively used nonlocal dispersal operator to replace the local diffusion term $d\Delta u$ is given by
 \[d(J * u-u)(x,t):=d\left(\int_{\Omega}J(x-y) u(y,t){\rm d}y-u(x,t)\right).\]

In the present paper we concern with the nonlocal dispersal monostable model in a given bounded domain $\oo\subset\R^N$ with seasonal succession. Let $\w>0$ and $0<\rho<1$ be constants which account for the period of seasonal succession and the duration of the bad season, respectlvely. In order to save space and facilitate writing, we denote
 \[\Sigma_i^l=(i\w,\,(i+\rho)\w],\;\;\;\Sigma_i^r=((i+\rho)\w,\,(i+1)\w],\;\;\;i=0,1,2,\cdots.\]
Then $\Sigma_i^l\cup\Sigma_i^r=(i\w,\,(i+1)\w]$ is one period, and $\bigcup_{i=0}^\yy(\Sigma_i^l\cup\Sigma_i^r)=(0,\infty)$. The corresponding model reads as follows
 \bes
 \left\{\begin{array}{llll}
 u_t=-\delta u,\,\;&(x,t)\in\bar\oo\times\Sigma_i^l,\\[1mm]
 u_t-d{\cal L}[u]=f(x,t,u),\,\;&(x,t)\in\bar\oo\times\Sigma_i^r,\\[1mm]
 u(x,0)=u_0(x), &x\in\bar\oo, \, i=0,1,2,\cdots,
 \end{array}\right.\lbl{1.1}
 \ees
where coefficient $\delta$ is a positive constant, and the nonlocal dispersal operator $\mathcal{L}[u]$ is given by
 \bess
 \mathcal{L}[u](x,t):=\int_{\oo}J(x-y)u(y,t){\rm d}y-u(x,t).
 \eess
The initial function $u_0$ satisfies
  \[u_0\in C(\bar\oo),\;~u_0>0~\text{ in }~\bar\oo.\]
In what follows, unless stated otherwise, we always take $i=0, 1, 2,\cdots$.
Assume that the dispersal kernel function $J\in C(\R^N)$ and satisfies
\begin{enumerate}[leftmargin=4em]
\item[{\bf(J)}]\; $J\ge 0$, $J(0)>0$, and $\int_{\R^N}J(x){\rm d}x=1$.
 \end{enumerate}
The growth term $f: \bar\oo\times(\cup_i\bar{\Sigma}_i^r)\times\mathbb{R}^+
\rightarrow\mathbb{R}$ is assumed to be continuous and satisfies
\begin{enumerate}[leftmargin=4em]
\item[\bf(F1)] $f(x,t,0)\equiv 0$, $f(x, t+\w,u)=f(x,t, u)$ for all $x\in\bar\oo$, $t\in\Sigma_i^r$, $u\ge 0$. The function $\frac{f(x,t,u)}u$ is strictly decreasing in $u\in(0,\yy)$ for all $(x,t)\in \bar\oo\times(\cup_{i}\Sigma_i^r)$;
\item[\bf(F2)] $f(x,t,u)$ is locally Lipschitz in $u\in\R^+$, i.e., for any $L>0$, there exists a constant $K=K(L)>0$
such that
\[\left|f(x,t,u_1)-f(x,t,u_2)\right|\le K|u_1-u_2|,\;\;\forall\; u_1, u_2\in [0, L],\;\; (x,t)\in \bar\oo\times(\cup_i\Sigma_i^r);\]
\item[\bf(F3)] There exists $K_0>0$ such that $f(x,t,K_0)=0$ for all $(x,t)\in \bar\oo\times(\cup_i\Sigma_i^r)$.
\end{enumerate}

In the case where $\rho=0$, problem \eqref{1.1} becomes the nonlocal dispersal monostable equation in time periodic and space heterogeneous environment. There have been many related works in this aspect. Please refer to \cite{RS2012, SV19, SLLY20} and the references therein.

The organization of this paper is as follows. In Section 2 we prove that the problem \eqref{1.1} has a unique bounded global positive solution. In Section 3, we study the positive solution of the corresponding time periodic problem. The section 4 is concerned with the dynamical properties of \qq{1.1}. The principal eigenvalue is determined firstly, and the long time behaviors of solution are given secondly. In section 5 we are mainly consider problem \qq{1.1} in the whole space. This paper can be regarded as the generalization of the nonlocal dispersal monostable equation without seasonal succession (\cite{RS2012, SX15, SV19, SLLY20}) and the local diffusive monostable equation with seasonal succession (\cite{MZ16, PZ13, WZZjde2021, WZZnon2022, ZZ13}).

\section{Existence and uniqueness of global solution} \lbl{s2}
\setcounter{equation}{0} {\setlength\arraycolsep{2pt}

In this section, we focus on the existence, uniqueness and boundedness of the global solution of problem \eqref{1.1}. We first prove the following maximum principle.

\begin{lem}\lbl{l2.1}{\rm(Maximum principle)} Let $c\in L^\infty(\oo\times
\cup_i\Sigma_i^r)$, and $v\in C(\bar\oo\times[0,\yy))\cap C^{0,1}(\bar\oo\times\bar{\Sigma}_i^l)\cap C^{0,1}(\bar\oo\times\bar{\Sigma}_i^r)$ satisfy
 \bes
 \left\{\begin{array}{lll}
 v_t\geq -\delta v, \,\;&(x,t)\in\bar\oo\times\Sigma_i^l,\\[1mm]
 v_t-d\mathcal{L}[v]\geq c(x,t)v, \,\;&(x,t)\in\bar\oo\times\Sigma_i^r,\\[1mm]
  v(x,0)\geq 0, &x\in\bar\oo.
 \end{array}\right.
 \lbl{2.3}\ees
Then $v\geq 0$ in $\bar\oo\times[0,\yy)$. Moreover, if $v(x_0,0)>0$ for some $x_0\in\oo$, then $v(x_0, t)>0$ in $[0,\rw]$ and $v>0$ in $\bar\oo\times(\rw,\yy)$.
\end{lem}

\begin{proof} By the first equation of \qq{2.3} with $i=0$ we have
 \[v(x,t)\geq e^{-\delta t}v(x,0)\ge 0,\;\; (x,t)\in\bar\oo\times\bar{\ooo}_0^l.\]
Then $v$ satisfies
 \bess
 \left\{\begin{array}{lll}
  v_t-d\mathcal{L}[v]\geq c(x,t)v,\; &(x,t)\in\bar\oo\times\ooo_0^r,\\[1.5mm]
  v(x,\rho\w)\geq e^{-\delta\rho\w}v(x,0)\ge 0,\;\; &x\in\bar\oo.
 \end{array}\right.
 \eess
As in the proof of \cite[Lemma 2.2]{CDLL2019} we can show that $v\ge 0$ in $\bar\oo\times\bar{\ooo}_0^r$. By the inductive process, $v\geq 0$ in $\bar\oo\times[0,\yy)$.

Assume $v(x_0,0)>0$ for some $x_0\in\oo$. Then $v(x_0,t)\geq e^{-\delta t}v(x_0,0)>0$ in $\bar{\ooo}_0^l$. Hence, $v(x_0,\rw)>0$. As $v(x,\rw)\ge 0$ for $x\in\bar\oo\backslash \{x_0\}$, following the proof of \cite[Lemma 2.2]{CDLL2019} we can show that $v>0$ in $\bar\oo\times\ooo_0^r$. Utilizing the inductive method, one has $v>0$ in $\bar\oo\times(\rw,\yy)$.
\end{proof}

As a consequence of Lemma \ref{l2.1} we have the following comparison principle.

\begin{lem}\lbl{l2.2}{\rm(Comparison principle)} Let $\bar u, \ud u\in C(\bar\oo\times[0,\yy))\cap C^{0,1}(\bar\oo\times\bar{\Sigma}_i^l)\cap C^{0,1}(\bar\oo\times\bar{\Sigma}_i^r)$ be non-negative bounded functions and satisfy
 \bess
 \left\{\begin{array}{lll}
 \bar u_t\geq -\delta \bar u, \,\;&(x,t)\in\bar\oo\times\Sigma_i^l,\\[1mm]
 \bar u_t-d\mathcal{L}[\bar u]\geq f(x,t,\bar u), \,\;&(x,t)\in\bar\oo\times\Sigma_i^r
 \end{array}\right.
 \eess
and
\bess
 \left\{\begin{array}{lll}
 \ud u_t\leq-\delta \ud u, \,\;&(x,t)\in\bar\oo\times\Sigma_i^l,\\[1mm]
 \ud u_t-d\mathcal{L}[\ud u]\leq f(x,t,\ud u), \,\;&(x,t)\in\bar\oo\times\Sigma_i^r,
 \end{array}\right.
 \eess
respectively. If $\bar u(x,0)\ge\ud u(x,0)$ in $\bar\oo$, then $\bar u(x,t)\ge\ud u(x,t)$ in $\bar\oo\times[0,\yy)$. Moreover, if $\bar u(x_0,0)>\ud u(x_0,0)$ for some $x_0\in\oo$, then $\bar u(x_0,t)>\ud u(x_0,t)$ in $[0,\rw]$ and $\bar u(x,t)>\ud u(x,t)$ in $\bar\oo\times(\rw,\yy)$.
\end{lem}

\begin{proof} Set $v(x,t)=\bar u(x,t)-\ud u(x,t)$.
Then $v$ satisfies
 \bess
 \left\{\begin{array}{lll}
 v_t\geq -\delta v, \,\;&(x,t)\in\bar\oo\times\Sigma_i^l,\\[1mm]
  v_t-d\mathcal{L}[v]\geq f(x,t, \bar u)-f(x,t,\ud u)\geq c(x,t)v,\; &(x,t)\in\bar\oo\times
  \ooo_i^r,\\[1.5mm]
  v(x,0)\geq 0, &x\in\bar\oo,
 \end{array}\right.
 \eess
 where $c(x,t)=-K$ when $\bar u(x,t)\geq \underline u(x,t)$, $c(x,t)=K$ when $\bar u(x,t)<\underline u(x,t)$, and $K$ is given by {\bf(F2)} with $L=\|\bar u\|_{L^\infty(\oo\times[0,\infty))}+\|\ud u\|_{L^\infty(\oo\times[0,\infty))}$. Clearly, $c\in L^\infty(\oo\times\cup_i\ooo_i^r)$. The desired conclusions can be deduced by Lemma \ref{l2.1}.
\end{proof}

\begin{theo}\lbl{th2.1}\, The problem \qq{1.1} has a unique global solution
 \bess
 u\in C(\bar\oo\times[0,\yy))\cap C^{0,\infty}(\bar\oo\times\bar{\Sigma}_i^l)\cap C^{0,1}(\bar\oo\times\bar{\Sigma}_i^r),\eess
and
 \bess
0<u(x,t)\le \max\left\{\max_{\bar\oo}u_0, ~K_0\right\}~\mbox{in}\;\; \bar\oo\times[0,\yy),
\eess
where $K_0$ is defined in the assumption {\bf(F3)}.
 \end{theo}

\begin{proof} {\it Step 1}. Define
  \[u^{0,l}(x,t)=e^{-\delta t}u_0(x), \,\, (x,t)\in\bar\oo\times\bar{\ooo}_0^l.\]
Then $u^{0,l}\in C^{0,\yy}(\bar{\oo}\times\bar{\ooo}_0^l)$ and
$0<u^{0,l}(x,t)\le \dd\max_{\bar\oo}u_0$ for $x\in\bar\oo$ and $t\in\bar{\ooo}_0^l$.
Clearly, $u^{0,l}$ satisfies
 \bess
 \left\{\begin{array}{lll}
 u^{0,l}_t=-\delta u^{0,l},\;\; &(x,t)\in\bar\oo\times\bar{\ooo}_0^l,\\[1mm]
 u^{0,l}(x,0)=u_0(x),\;\;  &x\in\bar\oo.
 \end{array}\right.
 \eess
It is well known that the problem
 \bess
 \left\{\begin{array}{llll}
 u_t-d{\cal L}[u]=f(x,t,u),\;\; &(x,t)\in\bar\oo\times\bar{\ooo}_0^r,\\[1mm]
 u(x,\rho\w)=u^{0,l}(x,\rho\w), &x\in\bar\oo
 \end{array}\right.
 \eess
has a unique positive solution $u^{0,r}\in C^{0,1}(\bar{\oo}\times\bar{\ooo}_0^r)$, and
 \bess
 0<u^{0,r}(x,t)\leq\max\left\{\max_{\bar\oo}u^{0,l}(\cdot,\rho\w),\, K_0\right\}\leq \max\left\{\max_{\bar\oo}u_0,\, K_0\right\}\;\;\;\text{in}\;\;\bar\oo\times\bar{\ooo}_0^r. \eess
Define
 \bess
 u^0(x,t)=\left\{\begin{array}{ll}
 u^{0,l}(x,t)\;\;\;\text{in} \;\; \bar{\oo}\times\bar{\ooo}_0^l,\\[1mm]
 u^{0,r}(x,t)\;\;\;\text{in} \;\;\bar{\oo}\times\bar{\ooo}_0^r.
 \end{array}\right.
  \eess
It then follows that
  \bess
 u^0\in C(\bar{\oo}\times[0,\w])\cap C^{0,\yy}(\bar{\oo}\times\bar{\ooo}_0^l)\cap C^{0,1}(\bar{\oo}\times\bar{\ooo}_0^r),\eess
 and
 \bess
  0<u^0(x,t)\leq \max\left\{\max_{\bar\oo}u_0, K_0\right\}\;\;\;\text{in}\;\;\bar\oo\times[0,\w].
 \eess

{\it Step 2}. For the above obtained function $u^0$, we set
  \[u^{1,l}(x,t)=e^{-\delta(t-\w)}u^0(x,\w),\,\; (x,t)\in\bar\oo\times\bar\ooo_1^l.\]
Then $ u^{1,l}\in C^{0,\yy}(\bar{\oo}\times\bar{\ooo}_1^l)$, and satisfies
 \bess
 \left\{\begin{array}{lll}
 u^{1,l}_t=-\delta u^{1,l},\;\; &(x,t)\in\bar\oo\times\bar\ooo_1^l,\\[1mm]
 u^{1,l}(x,\w)=u^0(x,\w),\;\;  &x\in\bar\oo,
 \end{array}\right.
 \eess
and
 \[0<u^{1,l}(x,t)\leq u^0(x,\w)\leq \max\left\{\max_{\bar\oo}u_0,K_0\right\} \;\;\;\mbox{in}\;\; \bar\oo\times\bar\ooo_1^l.\]
Similarly, the problem
 \bess
 \left\{\begin{array}{llll}
 u_t-d{\cal L}[u]=f(x,t,u), &(x,t)\in\bar\oo\times\bar\ooo_1^r,\\[1mm]
 u(x,(1+\rho)\w)=u^{1,l}(x,(1+\rho)\w), \;\; &x\in\bar\oo
 \end{array}\right.
 \eess
has a unique positive solution $u^{1,r}\in C^{0,1}(\bar{\oo}\times\bar{\ooo}_1^r)$, and
$$0<u^{1,r}(x,t)\leq \max\left\{\max_{\bar\oo}u^{1,l}(\cdot,(1+\rho)\w),K_0\right\}\leq  \max\left\{\max_{\bar\oo}u_0,K_0\right\} \;\;\;\mbox{in}\;\; \bar\oo\times\bar\ooo_1^r.$$
Define
 \bess\begin{array}{lll}
u^1(x,t)=\left\{\begin{array}{ll}
 u^{1,l}(x,t)\;\;\;&\text{in}\;\; \bar{\oo}\times\bar{\ooo}_1^l,\\[1mm]
 u^{1,r}(x,t)\;\;\;&\text{in}\;\; \bar{\oo}\times\bar{\ooo}_1^r.
 \end{array}\right.
  \end{array}\eess
It then follows that
 \bess
u^1\in C(\bar{\oo}\times [\w, 2\w])\cap C^{0,\yy}( \bar{\oo}\times\bar{\ooo}_1^l)
\cap C^{0,1}( \bar{\oo}\times\bar{\ooo}_1^r),
 \eess
and
\[0<u^1\leq \max\left\{\max_{\bar\oo}u_0,K_0\right\}
\;\;\;\text{in}\;\;\bar\oo\times[\w,2\w].\]

{\it Step 3}. Repeating this process, we can get a function $u$ which solves \eqref{1.1}. From the above construction we see that $u$ has the desired properties.

{\it Step 4}. The uniqueness of solutions is a consequence of the comparison principle (Lemma \ref{l2.2}).
\end{proof}

\section{Time periodic problem} \lbl{s3}
\setcounter{equation}{0} {\setlength\arraycolsep{2pt}

In this section, we mainly discuss the existence and uniqueness of positive time periodic solution to the following problem
 \bes
 \left\{\begin{array}{lll}
 U_t=-\delta U,  &(x,t)\in\bar\oo\times(0,\rw],\\[1mm]
 U_t-d\mathcal{L}[U]=f(x,t,U),\; &(x,t)\in\bar\oo\times(\rw,\w],\\[1mm]
 U(x,0)=U(x,\w),\;\;&x\in\bar\oo.
 \end{array}\right.\lbl{3.1}
 \ees
We first define an upper (a lower) solution of \eqref{3.1} as follows.

\begin{defi}\label{d3.1}\ Let $U\in C(\bar\oo\times[0,\w])\cap C^{0,\infty}(\bar\oo\times [0,\rw])\cap C^{0,1}(\bar\oo\times[\rw,\w])$  be a non-negative function. Such a function $U$ is referred to as an upper solution $($a lower solution$)$ of \eqref{3.1} if
 \bess\left\{\begin{array}{lll}
 U_t \geq\, (\leq)\,-\delta U,  \; &(x,t)\in\bar\oo\times(0,\rw],\\[1mm]
 U_t-d\mathcal{L}[U] \geq\, (\leq)\,f(x,t,U), \; &(x,t)\in\bar\oo\times(\rw,\w],\\[1mm]
 U(x,0)\geq\,(\leq) \, U(x,\w),\;\;&x\in\bar\oo.
 \end{array}\right.
 \eess
If $U$ is both the upper solution and the lower solution of \eqref{3.1}, we call that $U$ is a solution of \eqref{3.1}.
\end{defi}

Before give the main result of this section, we give the following two lemmas.

\begin{lem}\lbl{l3.1} Let $U\in L^\yy(\oo\times[0,\w])$ be a non-negative function and satisfy \eqref{3.1} pointwisely. Then $U\in C(\bar\oo\times[0,\w])\cap C^{0,\yy}(\bar\oo\times[0,\rw])\cap C^{0,1}(\bar\oo\times[\rw,\w])$.
 \end{lem}

\begin{proof} Since $U\in L^\yy(\oo\times[0,\w])$, we see that $U(\cdot,t)$ is Lipschitz continuous with respect to $t\in[0,\w]$, and there exists a constant $K_*>0$ such that
 \bes
 |U(x,t_1)-U(x,t_2)|\le K_*|t_1-t_2|,\;\;\forall\; x\in\boo,\; t_1, t_2\in[0,\w].
 \lbl{3.2}\ees

In the following we shall prove that $U(\cdot,t)$ is continuous in $x\in\boo$. In fact, for the fixed $x_0\in\bar\oo$, we denote $U^0(t)=U(x_0,t)$. Consider the following auxiliary time periodic problem
\bes
 \left\{\begin{array}{lll}
 V_t=-\delta V, \; &t\in(0,\rw],\\[1mm]
 V_t=-dV+f(x_0,t,V)+d\dd\int_{\oo}J(x_0-y)U(y,t){\rm d}y,\;&t\in(\rw,\w],\\[1mm]
 V(x_0,0)=V(x_0,\w).
 \end{array}\right.
 \lbl{3.3}\ees
Clearly, $U^0(t)$ is the unique solution of \eqref{3.3} because of such a problem has at most one solution. Given a sequence $\{x_n\}$ satisfying $x_n\to x_0$ as $n\to\infty$, and denote $U^n(t)=U(x_n, t)$. Then $U^n(t)$ is the unique solution of \eqref{3.3} with $x_0$ replaced by $x_n$. The estimate \qq{3.2} shows that $\{U^n(t)\}$ is uniformly bounded and equi-continuous on $[0,\w]$. By the Arzela-Ascoli's Theorem, there exist a subsequence $\{n_j\}$ of $\{n\}$ and a function $U^*\in C([0,\w])$ such that $U^{n_j}\to U^*$ in $C([0,\w])$. Letting $n_j\to\infty$ in the integral form of \qq{3.3} and using the continuities of $J(x)$ and $f(x,\cdot,\cdot)$ in $x$, we have that $U^*(t)$ satisfies \qq{3.3}. The uniqueness of solutions of \qq{3.3} implies $U^*(t)=U^0(t)$. That is, $\dd\lim_{{n_j}\to \yy}U(x_{n_j},t)=U(x_0,t)$ in $C([0,\w])$. This limit also holds for the whole sequences $\{x_n\}$ by the uniqueness of limit.

In conclusion, $U\in C(\bar\oo\times[0,\w])$. On the other hand, it is easy to see from the differential equations of \qq{3.1} that $U\in C^{0,\yy}(\bar\oo\times[0,\rw])\cap C^{0,1}(\bar\oo\times[\rw,\w])$.
\end{proof}

A non-negative function $U\in L^\yy(\oo\times[0,\w])$ satisfying \eqref{3.1} pointwisely can be regarded as a {\it weak} solution of \qq{3.1}. A non-negative function $U\in C(\bar\oo\times[0,\w])\cap C^{0,\infty}(\bar\oo\times[0,\rw])\cap C^{0,1}(\bar\oo\times[\rw,\w])$ satisfying \qq{3.1} in the classical sense is called a {\it classical} solution of \qq{3.1}. Lemma \ref{l3.1} shows that {\it weak} solution and {\it classical} solution are equivalent.

\begin{lem}\lbl{l3.2}{\rm(Comparison principle)} Let $\bar U$ and $\ud U$ be the upper and  lower solutions of \eqref{3.1}, respectively. If $\bar U$ and $\ud U$ are positive in $\bar\oo\times[0,\w]$, then $\bar U\ge \ud U$ in $\bar\oo\times[0,\w]$.
\end{lem}

\begin{proof}\, This proof is inspired by \cite[Proposition 4.1]{SV19}. As $\bar U$ and $\ud U$ are positive in $\bar\oo\times[0,\w]$, we can find a constant $\sigma>0$ such that the function $V=\bar U-\sigma\ud U$ satisfies $V\ge 0$ in $\bar\oo\times[0,\w]$ and $V(x_0,t_0)=0$ for some $(x_0,t_0)\in\bar\oo\times[0,\w]$. If $\sigma\ge 1$, then the conclusion holds. We assume $\sigma<1$. By the condition {\bf(F1)} we see that $V$ satisfies
 \bess
 V_t &\geq&-\delta V,  \;\; (x,t)\in\bar\oo\times(0,\rw],\\[1mm]
 V_t-d\mathcal{L}[V]&\geq& f(x,t,\bar U)-\sigma f(x,t,\ud U)>f(x,t,\bar U)
 -f(x,t,\sigma\ud U),\;\; (x,t)\in\bar\oo\times(\rw,\w],\\[1mm]
  V(x,0)&\geq& V(x,\w),\;\;x\in\bar\oo.
  \eess
It then follows that
 \[V(x_0,t)\ge e^{-\delta t}V(x_0,0),\;\; t\in(0,\rw].\]
If $t_0\in(\rw,\w]$, then $V_t(x_0,t_0)\le 0$. However,
 \[V_t(x_0,t_0)>d\mathcal{L}[V](x_0,t_0)=d\int_{\oo}J(x_0-y)V(y,t_0){\rm d}y\ge 0,\]
which is a contradiction. This also shows that $V>0$ in $\bar\oo\times(\rw,\w]$. If $t_0\in[0,\rw]$, then $0=V(x_0,t_0)\ge e^{-\delta t_0}V(x_0,0)$, i.e., $V(x_0,0)\le 0$. So $V(x_0,\w)\le 0$. By our assumption, $V(x_0,\w)=0$. This becomes the case that $t_0\in(\rw,\w]$, and we have obtained a contradiction.
\end{proof}

Now, we shall use the monotone iterative method to prove the existence of solution of \eqref{3.1}.

\begin{theo}\label{th3.1}{\rm(The upper and lower solutions method)}\, Suppose that $\underline U$ is a positive lower solution of \eqref{3.1}. Then the problem \eqref{3.1} has a unique solution $U$ which satisfies
 $$\underline{U}\leq U\leq K_0\ ~~\mbox{in}\, \ \bar\oo\times[0,\w],$$
where $K_0$ is given in the assumption {\bf(F3)}.
\end{theo}

\begin{proof}\,{\it Step 1}. It is obvious that $\bar U=K_0$ is a positive upper solution of \eqref{3.1}. And then $\bar U\ge \ud U$ in $\bar\oo\times[0,\w]$ by Lemma \ref{l3.2}.  We consider the following initial value problems
  \bes
 \left\{\begin{array}{lll}
 \ud U_{1t}=-\delta\ud U_1, \,\;&(x,t)\in\bar\oo\times(0,\rw],\\[1mm]
 \ud U_{1t}-d\mathcal{L}[\ud U_1]+K \ud U_1=f(x,t,\ud U)+K\ud U,\;\; &(x,t)\in\bar\oo\times(\rw,\w],\\[1mm]
 \ud U_1(x,0)=\ud U(x,\w),\;\;&x\in\bar\oo
 \end{array}\right.\lbl{3.4}
 \ees
 and
  \bes
 \left\{\begin{array}{lll}
 \ud U_{kt}=-\delta \ud U_k,  &(x,t)\in\bar\oo\times(0,\rw],\\[1mm]
 \ud U_{kt}-d\mathcal{L}[\ud U_k]+K\ud U_k=f(x,t,\ud U_{k-1})+K\ud U_{k-1},\;\;&(x,t)\in\bar\oo\times(\rw,\w],\\[1mm]
 \ud U_k(x,0)=\ud U_{k-1}(x,\w),\;\;&x\in\bar\oo, \;\;k=2,3,\cdots,
 \end{array}\right.
 \lbl{3.5}\ees
where $K$ is defined as in the assumption {\bf(F2)} with $L=K_0+1$. Adapting the method of Theorem \ref{th2.1}, we can easily see that \eqref{3.4} has a unique positive solution $\ud U_1$. Applying the inductive method, problem \eqref{3.5} admits a unique positive solution $\ud U_k$.

 {\it Step 2}. We claim that
\bes
\ud U\leq \ud U_1\leq \cdots \leq \ud U_{k-1}\leq \ud U_k\leq \cdots\leq\bar U,\ \ \ (x,t)\in\bar\oo\times[0,\w].
\lbl{3.6}\ees
In fact, let $W=\ud U_1-\ud U$, a direct calculation shows that $W(x,0)=\ud U(x,\w)-\ud U(x,0)\ge 0$ in $\bar\oo$, and
\bess
 W_t+\delta W&=&\ud U_{1t}+\delta \ud U_1-(\ud U_t+\delta \ud U)\geq 0,\;\; (x,t)\in\bar\oo\times(0,\rw],\\[1mm]
 W_t-d\mathcal{L}[W]+KW&=&\ud U_{1t}-d\mathcal{L}[\ud U_1]+K\ud U_1-(\ud U_t-d\mathcal{L}[\ud U]+K\ud U)\\[1mm]
 &\geq& f( x,t,\ud U)+K\ud U-(f(x,t,\ud U)+K\ud U)=0,\;\; (x,t)\in\bar\oo\times(\rw,\w].
 \eess
Utilizing Lemma \ref{l2.1}, we have $W\geq 0$ in $\bar\oo\times[0,\w]$. That is $\ud U\leq \ud U_1$ in $\bar\oo\times[0,\w]$. Particularly,
\bess
\ud U_1(x,0)=\ud U(x,\w)\leq \ud U_1(x,\w), \ \ x\in \bar\oo.
\eess
We then apply Lemma \ref{l2.1} to $\ud U_1$ and $\ud U_2$  to derive that $\ud U_1\leq \ud U_2$ in $\bar\oo\times[0,\w]$. Thus, $\ud U_2(x,0)=\ud U_1(x,\w)\leq \ud U_2(x,\w)$ in $\bar\oo$. Repeating the above process we obtain
\bes
\ud U_{k-1}\leq \ud U_k, \;\; \forall\; (x,t)\in\bar\oo\times[0,\w],\;\;k\geq 1.
\lbl{3.7}\ees

Let $\tilde W=\bar U-\ud U_1$. Then $\tilde W(x,0)=\bar U(x,0)-\ud U(x,\w)\geq 0$ for $x\in\bar\oo$, and
\bess
 \tilde W_t+\delta \tilde W&=&\bar U_{t}+\delta \bar U-(\ud U_{1t}+\delta \ud U_1)\ge 0,\;\; (x,t)\in\bar\oo\times(0,\rw],\\[1mm]
 \tilde W_t-d\mathcal{L}[\tilde W]+K\tilde W&=&\bar U_{t}-d\mathcal{L}[\bar U]+K\bar U-(\ud U_{1t}-d\mathcal{L}[\ud U_1]+K\ud U_1)\\[1mm]
 &\ge&f(x,t,\bar U)+K\bar U-(f(x,t,\ud U)+K\ud U)\geq 0,\,\; (x,t)\in\bar\oo\times(\rw,\w].
 \eess
In view of Lemma \ref{l2.1}, we have $\tilde W\geq 0$, that is, $\ud U_1\leq \bar U$ in $\bar\oo\times[0,\w]$. Taking advantage of the inductive method, we obtain $\ud U_k\le\bar U$ in $\bar\oo\times[0,\w]$ for all $k\geq 2$. This, together with \eqref{3.7}, yields that \qq{3.6} holds.

{\it Step 3}.\, By the monotonicity and boundedness of $\{\ud U_k\}$, there exists a positive function $U_*$ such that $\ud U_k\to U_*$ pointwisely in $\bar\oo\times [0,\w]$ as $k\to\infty$. As $\ud U_k(x,0)=\ud U_{k-1}(x,\w)$, we have
  \[U_*(x,0)=U_*(x,\w).\]
It follows from \qq{3.5} that $\ud U_k$ satisfies
 \bess
 \ud U_k(x,t)&=&e^{-\delta t}\ud U_{k-1}(x,\w),\;\; (x,t)\in\bar\oo\times(0,\rho\w],\\
  \ud U_k(x,t)&=&\ud U_k(x,\rw)+\dd\int_{\rho\w}^t\big[d\mathcal{L}[\ud U_k] -K(\ud U_k-\ud U_{k-1})+f(x,s,\ud U_{k-1})\big]{\rm d}s,\;\; (x,t)\in\bar\oo\times(\rw,\w].
 \eess
Taking $k\to\yy$ we then obtain that $U_*$ satisfies, by the {\it dominated convergence theorem},
 \bess
 U_*(x,t)&=&e^{-\delta t}U_*(x,\w),\;\; (x,t)\in\bar\oo\times(0,\rho\w],\\
 U_*(x,t)&=&U_*(x,\rw)+\dd\int_{\rho\w}^t\big[d\mathcal{L}[U_*]+f(x,s,U_*)\big]{\rm d}s,\;\; (x,t)\in\bar\oo\times(\rw,\w].
 \eess
It is easy to see that $U_*$ satisfies the first two differential equations of \qq{3.1}. Furthermore, $U_*\in C(\bar\oo\times[0,\w])\cap C^{0,\infty}(\bar\oo\times[0,\rw])\cap C^{0,1}(\bar\oo\times[\rw,\w])$ by Lemma \ref{l3.1}.
Thus $U_*$ is a positive solution of \qq{3.1}.

{\it Step 4}. The uniqueness of positive solution can be deduced by Lemma \ref{l3.2}.
\end{proof}

\section{Dynamical properties of \qq{1.1}}
\setcounter{equation}{0} {\setlength\arraycolsep{2pt}

 From now on, we always assume that $f$ is continuously differentiable with respect to $u\ge 0$ and
 \[f_u(x,t,0)=a(t)+b(x),\]
where $a\in C(\bar{\Sigma}_i^r)$, $a(t+\w)=a(t)$ for $t\in\Sigma_i^r$ and $b\in C(\bar\oo)$. Let
$a^*(t)=a(t)-\bar a$ and $b^*(x)=b(x)+\bar a$, where
\[\bar{a}=\frac{1}{(1-\rho)\w}\int^{w}_{\rho\w}a(s){\rm d}s.\]
Obviously, $\bar{a}^*=0$. Denote $a^*$ and $b^*$ by $a$ and $b$, respectively. Thus, we may assume that $\bar a=0$.

To study the dynamical properties of \qq{1.1}, we first deal with the $\w$-periodic (in time $t$) parabolic eigenvalue problem
\bes
 \left\{\begin{array}{lll}
 \vp_t=-\delta \vp+\lm\vp,  &(x,t)\in\bar\oo\times(0,\rw],\\[1mm]
 \vp_t-d\mathcal{L}[\vp]-[a(t)+b(x)]\vp=\lm\vp,\;\;&(x,t)\in\bar\oo\times(\rw,\w],\\[1mm]
\vp(x,0)=\vp(x,\w),\;\;&x\in\bar\oo.
 \end{array}\right.\lbl{4.1}
 \ees

Define
 $$\lambda_p(\oo)=\sup_{\phi\in C(\bar\Omega),\;\phi>0\;{\rm in}\;\bar\Omega}\;\inf_{x\in\Omega}\frac{-d\mathcal{L}[\phi](x)-b(x)\phi(x)}{\phi(x)},$$
and set $b_m=\min_{\bar\Omega}b$,  $b_M=\max_{\bar\Omega}b$. We assume that one of the following holds:
\begin{enumerate}[leftmargin=4em]
\item[\bf(H1)] $J$ is supported in $B_\gamma(0)$ for some $\gamma>0$, and $(b_M-b)^{-1}\not\in L^1(\Omega)$;
\item[\bf(H2)] $\dd\inf_{y\in\Omega}\int_\Omega J(x-y){\rm d}x>\frac{b_M-b_m}d$;
\item[\bf(H3)]there exists a measurable subset $A\subset\Omega$ such that
 $\dd\inf_{y\in A}\int_A\frac{J(x-y)}{b_M-b(x)}{\rm d}x>\frac 1d$.
   \end{enumerate}
Then $\lambda_p(\oo)$ is a principal eigenvalue of
 \[ -d\mathcal{L}[\phi]-b(x)\phi=\lm\phi,\;\; x\in\bar\oo\]
with positive eigenfunction $\phi_p\in C(\bar\Omega)$
(see \cite{C2010, SX15, BCV2016, LCWdcds17}, for example), and some properties of $\lambda_p(\oo)$ were given in there.

Now we investigate the principal eigenvalue of \qq{4.1}. Set
 $$\sigma(t)=\delta\;\;\text{in}\;\;(0,\rw], \;\;\;
 \sigma(t)=\lambda_p(\oo)-a(t)\;\;\text{in}\;\;(\rw,\w],$$
and
\bess
\lambda_p^\w(\oo)=\delta\rho+(1-\rho)\lambda_p(\oo),\;\;
 \vp_p(x,t)=\exp\left\{\lambda_p^\w(\oo)t-\int_0^t\sigma(s){\rm d}s\right\}\phi_p(x).
\eess
It is easy to verify that $\vp_p$ satisfies \qq{4.1} with $\lm=\lambda_p^\w(\oo)$.
This shows that $\lambda_p^\w(\oo)$ is a principal eigenvalue of \qq{4.1} with positive eigenfunction $\vp_p\in C(\bar\Omega\times[0,\w])$.

\begin{theo}\lbl{th4.1} Let $u(x,t)$ be the unique positive solution of \eqref{1.1}.

{\rm(i)} If $\lambda_p^\w(\oo)\geq 0$, then $0$ is the unique non-negative solution of \eqref{3.1}, and
 \bes
 \lim_{t\to \infty}u(x,t)=0\;\;\;\text{in}\;\;C(\bar\oo).
 \lbl{4.2}\ees

{\rm(ii)} If $\lambda_p^\w(\oo)<0$, then \eqref{3.1} has a unique positive solution $U(x,t)$, and
 \bes
 \lim_{i\to \infty}u(x,t+i\w)=U(x,t)\;\;\;\text{in}\;\;C(\bar\oo\times[0,\w]).
 \lbl{4.3}\ees
\end{theo}

\begin{proof} (i) Assume $\lambda_p^\w(\oo)\geq 0$. We first prove that $0$ is the unique non-negative solution of \eqref{3.1}. Assume on the contrary that $U\ge 0,\,\not\equiv 0$ is a solution of \eqref{3.1}.
It is easy to show that $U(x,t)>0$ in $\bar\oo\times[0,\w]$ by the maximum principle (Lemma \ref{l2.1}). By the condition {\bf(F1)}, we obtain
 $$f(x,t, U(x,t))<f_u(x,t, 0)U(x,t)=[a(t)+b(x)]U(x,t)\;\;\; \text{in}\;\;\bar\oo\times[\rw,\w].$$
Thus, $U$ satisfies
 \bess
 \left\{\begin{array}{lll}
 U_t=-\delta U,  &(x,t)\in\bar\oo\times(0,\rw],\\[1mm]
 U_t-d\mathcal{L}[U]<[a(t)+b(x)]U,\;\;&(x,t)\in\bar\oo\times(\rw, \w],\\[1mm]
  U(x,0)=U(x,\w),&x\in\bar\oo.
 \end{array}\right.
 \eess
Since $(\lambda_p^\w(\oo),\vp_p)$ satisfies \qq{4.1} and $\lambda_p^\w(\oo)\geq 0$, it yields
 \bess
 \left\{\begin{array}{lll}
 \vp_{pt}\ge-\delta \vp_p, &(x,t)\in\bar\oo\times(0,\rw],\\[1mm]
 \vp_{pt}-d\mathcal{L}[\vp_p]-[a(t)+b(x)]\vp_p\ge 0, \;\;&(x,t)\in\bar\oo\times(\rw, \w],\\[1mm]
  \vp_p(x,0)=\vp_p(x,\w),\;\;&x\in\bar\oo.
 \end{array}\right.
 \eess
We can find a constant $\sigma>0$ such that the function $w=\vp_p-\sigma U$ satisfies $w\geq 0$ in $\bar\oo\times[0,\w]$, and $w(x_0,t_0)=0$ for some $(x_0,t_0)\in \bar\oo\times[0,\w]$. Then $w$ satisfies
\bes
 \left\{\begin{array}{lll}
 w_t\ge-\delta w,  &(x,t)\in\bar\oo\times(0,\rw],\\[1mm]
 w_t-d\mathcal{L}[w]-[a(t)+b(x)]w>0,\;\;&(x,t)\in\bar\oo\times(\rw, \w],\\[1mm]
  w(x,0)=w(x,\w),\;\;&x\in\bar\oo.
 \end{array}\right.
 \lbl{4.4}\ees
If $t_0\in(\rw,\w]$, then $w_t(x_0,t_0)\leq 0$. However, with the second inequality of \eqref{4.4},
 \[w_t(x_0,t_0)>d\mathcal{L}[w](x_0,t_0)\geq 0.\]
 This is a contradiction. Thus, $t_0\in[0,\rw]$ and $w>0$ in $\bar\oo\times(\rw, \w]$.
In view of the first inequality of \eqref{4.4}, it follows that
 \[0=w(x_0, t_0)\ge e^{-\delta t_0} w(x_0, 0)=e^{-\delta t_0}w(x_0,\w)>0.\]
This is a  contradiction. Therefore, $0$ is the unique non-negative solution of \eqref{3.1}.

Now we prove \qq{4.2}. The idea of this proof comes from \cite[Theorem 7.10]{Wpara21}. Take a constant $M>\max\big\{\max_{\bar\oo}u_0, ~K_0\big\}$. Then the function $\bar u=M$ satisfies
 \bes
 \left\{\begin{array}{llll}
 \bar u_t>-\delta\bar u,\,\;&(x,t)\in\bar\oo\times\Sigma_i^l,\\[1mm]
 \bar u_t-d{\cal L}[\bar u]\ge f(x,t,\bar u),\,\;&(x,t)\in\bar\oo\times\Sigma_i^r,\\[1mm]
 \bar u(x,0)>u_0(x), &x\in\bar\oo.
 \end{array}\right.\lbl{4.5}
 \ees
By Theorem \ref{th2.1}, the problem
\bess
 \left\{\begin{array}{lll}
 z_t=-\delta z, \,\;&(x,t)\in\bar\oo\times\Sigma_i^l,\\[1mm]
 z_t-d\mathcal{L}[z]=f(x,t,z),\;\; &(x,t)\in\bar\oo\times\Sigma_i^r,\\[1mm]
 z(x,0)=M,\;\;&x\in\bar\oo
 \end{array}\right.
 \eess
has a unique solution $z\in C(\bar\oo\times [0,\yy))\cap C^{0,\yy}(\bar\oo\times\bar{\Sigma}_i^l)\cap C^{0,1}(\bar\oo\times\bar{\Sigma}_i^r)$,
and $u\le z\leq M$ in $\bar\oo\times[0,\yy)$ by the comparison principle (Lemma \ref{l2.2}). Let $z^i(x,t)=z(x,t+i\w)$. In view of $f(x,t,u)$ is time periodic with period $\w$, it follows that $z^i(x,t)$ satisfies
 \bess
 \left\{\begin{array}{lll}
 z^i_t=-\delta z^i,  &(x,t)\in\bar\oo\times(0,\rw],\\[1mm]
 z^i_t-d\mathcal{L}[z^i]=f(x,t,z^i),\; &(x,t)\in\bar\oo\times(\rw, \w],\\[1mm]
  z^i(x,0)=z^{i-1}(x,\w),\;\;&x\in\bar\oo.
 \end{array}\right.
 \eess
Note that
 \begin{eqnarray*}
  z^1(x,0)=z(x,\w)\leq M=z(x,0), \;\; x\in\bar\oo.
 \end{eqnarray*}
We can adopt Lemma \ref{l2.2} to produce $z^1\leq z\le M$ in $\bar\oo\times[0,\w]$ which in turn  asserts
 $$ z^2(x,0)=z^1(x,\w)\leq z(x,\w)=z^1(x,0), \;\; x\in\bar\oo.$$
As above, $z^2\leq z^1\leq M$ in $\bar\oo\times[0,\w]$.
Applying the inductive method, we can prove that $z^i$ is monotonically decreasing in $i$ and $z^i\leq M$ in $\bar\oo\times[0,\w]$ for all $i$. Consequently, there exists a non-negative function $z^*$ such that $z^i\to z^*$ pointwisely in $\bar\oo\times[0,\w]$ as $i\to\infty$. Clearly, $z^*(x,0)=z^*(x,\w)$. Similar to the proof of Theorem \ref{th3.1} we can show that $z^*$ satisfies the first two differential equations of \qq{3.1}. And so $z^*\in C(\bar\oo\times[0,\w])\cap C^{0,\yy}(\bar\oo\times[0,\rho\w])\cap C^{0,1}(\bar\oo\times[\rho\w,\w])$ by Lemma \ref{l3.1}. Thus $z^*$ is a non-negative solution of \qq{3.1}, which implies that
$z^*=0$ since $0$ is the unique non-negative solution of \eqref{3.1}. Hence, $\dd\lim_{i\to\infty}z^i(x,t)=0$ uniformly in $\boo\times[0,\w]$ by Dini's Theorem.  Using the fact that $u(x,t+i\w)\le z^i(x,t)$ in $\bar\oo\times[0,\yy)$, we can show that \qq{4.2} holds.

{\rm(ii)} Assume $\lambda_p^\w(\oo)<0$. Let $\ud u(x,t)=\ep\vp_p(x,t)$ with $0<\ep\ll 1$. Then $\ud u>0$ in $\bar\oo\times[0,\w]$ and $\ud u\in C(\bar\oo)\times[0,\w]\cap C^{0,\infty}(\bar\oo\times[0,\rw])\cap C^{0,1}(\bar\oo\times[\rw,\w])$. Clearly, $\ud u(x,0)=\ud u(x,\w)$ for $x\in\bar\oo$, and
 \[\ud u_t=-\delta\ud u,\;\;\;(x,t)\in\bar\oo\times(0,\rw].\]
By the standard method we have that, when $\ep$ is very small,
\bess
 \ud u_t-d\mathcal{L}[\ud u]
 &=&f(x,t,\ud u)+[f_u(x,t,0)-f_u(x,t,\tau(x,t))+\lm_p^\w(\oo)]\ud u \\
 &\le& f(x,t,\ud u),\;\;
 (x,t)\in\bar\oo\times(\rw, \w],
 \eess
where $0\leq \tau \leq \ud u$. This indicates that $\ud u=\ep\vp_p$ is a positive lower solution of \qq{3.1}. Obviously, for any given constant $M>\max\big\{\max_{\bar\oo}u_0, ~K_0\big\}$, the function $\bar u=M$ is an upper solution of \qq{3.1}. It is derived by  Theorem \ref{th3.1} that \qq{3.1} has a unique positive solution $U$.

Now we prove \qq{4.3}. Noticing that $u_0(x)$ is positive and bounded, we can find $0<\ep\ll 1$ such that $\ep\vp_p(x,0)\le u_0(x)$ in $\boo$. Let $\ud z(x,t)$ and $\bar z(x,t)$ be the unique solutions of
 \bess
 \left\{\begin{array}{lll}
 \ud z_t=-\delta\ud z, \,\;&(x,t)\in\bar\oo\times\Sigma_i^l,\\[1mm]
 \ud z_t-d\mathcal{L}[\ud z]=f(x,t,\ud z),\;\; &(x,t)\in\bar\oo\times\Sigma_i^r,\\[1mm]
 \ud z(x,0)=\ep\vp_p(x,0),\;\;&x\in\bar\oo
 \end{array}\right.
 \eess
and
\bess
 \left\{\begin{array}{lll}
 \bar z_t=-\delta\bar z, \,\;&(x,t)\in\bar\oo\times\Sigma_i^l,\\[1mm]
 \bar z_t-d\mathcal{L}[\bar z]=f(x,t,\bar z),\;\; &(x,t)\in\bar\oo\times\Sigma_i^r,\\[1mm]
 \bar z(x,0)=M,\;\;&x\in\bar\oo,
 \end{array}\right.
 \eess
respectively. Then $\ud z(x,t)\le u(x,t)\le \bar z(x,t)$ by the comparison principle\;(Lemma \ref{l2.2}). Define
 \[\ud z^i(x,t)=\ud z(x,t+i\w),\;\;\; \bar z^i(x,t)=\bar z(x,t+i\w), \;\; \forall\;(x,t)\in\bar{\oo}\times[0,\w].\]
Similar to the discussions with respect to $z^i$ in the proof of \qq{4.2}, we can show that $\ud z^i$ and $\bar z^i$ are, respectively, increasing and decreasing in $i$, and there exist two functions $\ud z^*$ and $\bar z^*$ such that
 \[\lim_{i\to\yy}\ud z^i(x,t)=\ud z^*(x,t),\;\;\;\lim_{i\to\yy}\bar z^i(x,t)=\bar z^*(x,t)
 \;\;\;\text{in}\;\;C(\bar\oo\times[0,\w]),\]
and $\ud z^*(x,t)$ and $\bar z^*(x,t)$ are positive solutions of \qq{3.1}. By the uniqueness, $\ud z^*=\bar z^*=U$. Since $\ud z^i(x,t)\le u(x,t+i\w)
\le\bar z^i(x,t)$, the limit \qq{4.3} holds.
\end{proof}

\section{The problems in whole space}
\setcounter{equation}{0} {\setlength\arraycolsep{2pt}

In this section we study the problems in the whole space, i.e., the following problem
 \bes
 \left\{\begin{array}{llll}
 u_t=-\delta u,\,\;&(x,t)\in\R^N\times\Sigma_i^l,\\[1mm]
 u_t-d{\cal L}_{\R^N}[u]= f(x,t,u),\,\;&(x,t)\in\R^N\times\Sigma_i^r,\\[1mm]
 u(x,0)=u_0(x)>0, &x\in\R^N, \;i=0,1,2,\cdots
 \end{array}\right.\lbl{5.1}
 \ees
and the corresponding time periodic problem
 \bes
 \left\{\begin{array}{lll}
 U_t=-\delta U,  &(x,t)\in\R^N\times(0,\rw],\\[1mm]
 U_t-d\mathcal{L}_{\R^N}[U]= f(x,t,U),\; &(x,t)\in\R^N\times(\rw,\w],\\[1mm]
  U(x,0)=U(x,\w),\;\;&x\in\R^N,
 \end{array}\right.\lbl{5.2}
 \ees
where
 $$\mathcal{L}_{\R^N}[u](x,t)=\int_{\R^N}J(x-y)u(y,t){\rm d}y-u(x,t).$$
Assume that the condition {\bf(J)} holds and $J$ is supported in $B_\gamma(0)$ for some $\gamma>0$, $f$ satisfies the conditions {\bf(F1)}-{\bf(F3)} with $\oo=\R^N$, and $u_0\in C(\R^N)$ is bounded.

We first state the maximum principle in the whole space, which can be proved by adapting the main ideas of Lemma \ref{l2.1} and \cite[ Proposition 3.1]{RS2012}.
\begin{lem}\lbl{l5.1} Let $c\in L^{\infty}(\R^N\times[0,\yy))$, and $u\in C(\R^N\times[0,\yy))\cap C^{0,1}(\R^N\times\bar{\Sigma}_i^l)\cap C^{0,1}(\R^N\times\bar{\Sigma}_i^r)$ satisfy
 \bess
 \left\{\begin{array}{lll}
 u_t\geq -\delta u, \,\;&(x,t)\in\R^N\times\Sigma_i^l,\\[1mm]
 u_t-d\mathcal{L}_{\R^N}[u]\geq c(x,t)u, \,\;&(x,t)\in\R^N\times\Sigma_i^r,\\[1mm]
 u(x,0)\geq 0, &x\in\R^N.
 \end{array}\right.
\eess
Then $u\geq 0$ in $\R^N\times[0,\yy)$. Moreover, if $u(x,0)\not\equiv0$ in $\R^N$, then $u>0$ in $\R^N\times(\rw,\yy)$.
\end{lem}

To investigate the dependence of solutions of \eqref{5.1} on initial function $u_0$, we denote the solution of \eqref{5.1} by $u(x,t; u_0))$. For any given $u_0,v_0\in C(\R^N,\R)$ satisfying $u_0, v_0>0$ for $x\in\R^N$, we define
\bess
\Theta(u_0, v_0)=\inf\kk\{\ln\alpha:\frac1{\alpha}u_0(\cdot)\leq v_0(\cdot)\leq \alpha u_0(\cdot), \;\; \alpha\geq 1\rr\}.
\eess
When $t\in[0,\rho\w]$, direct calculation shows that  $u(x,t)=e^{-\delta t}u_0(x)$ for $x\in\R^N$. Thus,
\bess
\Theta(u(\cdot,t;u_0), u(\cdot,t; v_0))=\Theta(u_0, v_0), \ \ t\in[0,\rho\w].
\eess
When $t\in(\rho\w,\w]$, \eqref{5.1} is equivalent to
 \bess
 \left\{\begin{array}{lll}
u_t-d\mathcal{L}_{\R^N}[u]= f(x,t,u),\; &(x,t)\in\R^N\times(\rw,\w],\\[1mm]
u(x,\rho\w)=e^{-\delta\rho\w}u(x,0),\;\;&x\in\R^N,
 \end{array}\right.
 \eess
Similar to the proof of \cite[Proposition 5.1]{RS2012}, we have the following argument.
\begin{lem}\lbl{l5.2} For any given $u_0,v_0\in C(\R^N,\R)$ satisfying $u_0, v_0>0$ and $u_0\neq v_0$ for all $x\in\R^N$. When $t\in(\rho\w,\w]$, then $\Theta(u(\cdot, t; u_0), u(\cdot, t; v_0))$
strictly decreasing in time $t$.
\end{lem}

Next, same as the proof of Theorem \ref{th2.1}, and using Lemma \ref{l5.1}, we can prove the following theorem.

\begin{theo}\lbl{th5.1}\, The problem \qq{5.1} has a unique global solution
 \bess
 u\in C(\R^N\times[0,\yy))\cap C^{0,\infty}(\R^N\times\bar{\Sigma}_i^l)\cap C^{0,1}(\R^N\times\bar{\Sigma}_i^r),\eess
and
 \bess
0<u(x,t)\le \max\left\{\sup_{\R^N}u_0, ~K_0\right\}~\mbox{ in }\; \R^N\times[0,\yy),
\eess
where $K_0$ is defined in the assumption {\bf(F3)}.
 \end{theo}

Similar to Section 4, we assume that
 \[f_u(x,t,0)=a(t)+b(x),\;\;\text{with}\;\;\bar{a}=\frac{1}{(1-\rho)\w}\int^{w}_{\rho\w}a(s){\rm d}s=0,\]
where $b\in C(\R^N)\cap L^\yy(\R^N)$, $a\in C(\bar{\Sigma}_i^r)$ and $a(t+\w)=a(t)$ for $t\in\Sigma_i^r$. In the following it is assumed that
\begin{enumerate}
\item[\bf(H4)]there exists $R_0>0$ such that the condition {\bf(H1)} holds with $\oo=B_R(0)$ for all $R>R_0$.
\end{enumerate}
Here we mention that if for any $R>0$, there exists a positive constant $C(R)$ such that
 \[|b(x)-b(y)|\le C(R)|x-y|^N,\;\;\;\forall\;x,y\in B_R(0),\]
then the condition {\bf(H1)} holds.

To save spaces, we denote
  \[\lambda_p(B_R(0))=\lambda_p(R),\;\;\;\lambda_p^\w(B_R(0))=\lambda_p^\w(R).\]
Then $\lambda_p(R)$ is a principal eigenvalue of
 \[-d\kk(\int_{B_R(0)}J(x-y)\phi(y,t){\rm d}y-\phi(x,t)\rr)-b(x)\phi=\lm\phi,\;\; x\in B_R(0),\]
Utilizing \cite[Proposition 1.1(i) and Theorem 1.1]{C2010}, we know that $\lambda_p(R)$ is monotone decreasing with respect to $R$ and
  \[d-\sup_{x\in B_R(0)}\left(b(x)+d\int_{B_R(0)}J(x-y){\rm d}y\right)<\lambda_p(R)<d-\sup_{x\in B_R(0)}b(x).\]
Thus, the limit
 \[\lim_{R\to\yy}\lambda_p(R)=\lambda_p(\yy)\]
exists, and
  \[-\sup_{x\in \R^N}b(x)<\lambda_p(\infty)<d-\sup_{x\in\R^N}b(x).\]
Therefore, the principal eigenvalue $\lambda_p^\w(R)=\delta\rho+(1-\rho)\lambda_p(R)$ of \qq{4.1} with $\oo=B_R(0)$ is bounded and monotone decreasing in $R$, and the limit
 \[\lim_{R\to\yy}\lambda_p^\w(R)=\lambda_p^\w(\yy)\]
exists and is finite. Especially, when $b(x)=b$ is a constant, similar to the proof of \cite[Proposition 3.4]{CDLL2019} we have that $\lambda_p(R)\to -b$ as $R\to\infty$, and
$\lambda_p(R)\to d-b$ as $R\to 0$. It follows that
$$\lim_{R\to\infty}\lambda_p^\w(R)=\lambda_p^\w(\yy)=\delta\rho-(1-\rho)b, \ \
\lim_{R\to0}\lambda_p^\w(R)=\delta\rho-(1-\rho)(b-d).$$

\begin{theo}\lbl{th5.2}\, If $\lambda_p^\w(\yy)<0$, then problem \qq{5.2} has a unique  positive solution $U(x,t)$, and the unique solution $u(x,t)$ of \qq{5.1} satisfies
 \bes
 \lim_{i\to\yy}u(x, t+i\w)=U(x,t)\;\;\;\text{locally uniformly in}\;\;\R^N\times[0,\w].
 \lbl{5.3}\ees
 \end{theo}

\begin{proof}\, {\it Step 1: Existence}. As $\lambda_p^\w(\yy)<0$, we can find a $R_0\gg 1$ such that
 \[\lambda_p^\w(R)<0,\;\;\forall\; R\ge R_0.\]
For any monotone increasing sequence $\{R_n\}_{n=1}^{\infty}$ satisfying $R_1\geq R_0$ and $R_n\to\infty$ as $n\to\infty$, we consider the problems
 \bes
 \left\{\begin{array}{llll}
 u_t=-\delta u,\,\;&(x,t)\in \up{B_{R_n}(0)}\times\Sigma_i^l,\\[1mm]
 u_t-d{\cal L}_n[u]=f(x,t,u),\,\;&(x,t)\in\up{B_{R_n}(0)}\times\Sigma_i^r,\\[1mm]
 u(x,0)=u_0(x)>0, &x\in\up{B_{R_n}(0)},\;i=0,1,2,\cdots
 \end{array}\right.\lbl{5.4}
 \ees
and
 \bes
 \left\{\begin{array}{lll}
U_t=-\delta U,  &(x,t)\in\up{B_{R_n}(0)}\times(0,\rw],\\[1mm]
U_t-d\mathcal{L}_n[U]=f(x,t,U),\; &(x,t)\in\up{B_{R_n}(0)}\times(\rw,\w],\\[1mm]
U(x,0)=U(x,\w),\;\;&x\in\up{B_{R_n}(0)}, \;n\geq 1,
 \end{array}\right.\lbl{5.5}
 \ees
where
\bess
\mathcal{L}_n[u](x,t)=\int_{B_{R_n(0)}}J(x-y)u(y,t){\rm d}y-u(x,t).
\eess

In view of Theorems \ref{th2.1} and \ref{th4.1}, it follows that problems \qq{5.4} and \qq{5.5} have unique positive solutions $u_n(x,t)$ and $U_n(x,t)$, respectively, which satisfy
 \bes
 \lim_{i\to\yy}u_n(x, t+i\w)=U_n(x,t)\;\;\;\text{in}\;\;C(\up{B_{R_n}(0)}\times[0,\w]).
 \lbl{5.6}\ees
Let $M>\max\big\{\dd\sup_{\R^N}u_0, ~K_0\big\}$. By Lemmas \ref{l2.2},  and \ref{l3.2} we can show that
\bess
&0<u_n(x,t)\leq u_{n+1}(x,t)\leq M,\;\;
\forall\;(x,t)\in\up{B_{R_n}(0)}\times[0,\infty),\;n\geq 1;\\[1mm]
&0<U_n(x,t)\leq U_{n+1}(x,t)\leq M,\;\forall\;(x,t)\in\up{B_{R_n}(0)}\times[0,\w],\;n\geq 1.
\eess
Thus, there exist $u\in L^{\infty}(\R^N\times [0,\infty))$ and $U\in L^{\infty}(\R^N\times [0,\w])$ such that
\bess
&\dd\lim_{n\to\infty}u_n(x,t)=u(x,t)\;\;\mbox{point-wise in}\;\;\R^N\times [0,\infty);\\[1mm]
&\dd\lim_{n\to\infty}U_n(x,t)=U(x,t)\;\;\mbox{point-wise in}\;\;\R^N\times [0,\w].
\eess
Letting $n\to\infty$ in \eqref{5.4} and \eqref{5.5}, and using the dominated convergence theorem, we can show that $u$ and $U$ are, respectively, the positive solutions of \eqref{5.1} and \eqref{5.2}.

{\it Step 2: Uniqueness}. Suppose that $U_1(x,t)$ and $U_2(x,t)$ are two positive solutions of problem \eqref{5.2}. If $U_1\neq U_2$, we denote $(U_1(\cdot,0), U_2(\cdot,0))=:(\phi_0, \psi_0)$, it then follows from Lemma \ref{l5.2} that
$$\Theta(U(\cdot, \w; \phi_0), U(\cdot, \w; \psi_0))<\Theta(U(\cdot, t; \phi_0), U(\cdot, t; \psi_0))<\Theta(U(\cdot, \rho\w; \phi_0), U(\cdot, \rho\w; \psi_0)=\Theta(\phi_0, \psi_0)$$
for $t\in(\rho\w,\w]$. Notice that
$$(U(\cdot, \w; \phi_0), U(\cdot, \w; \psi_0))=(U_1(\cdot,\w), U_2(\cdot,\w))=(\phi_0, \psi_0).$$ This is a contradiction. Therefore, $U_1=U_2$.

{\it Step 3}. By the Dini's theorem we have that
\bess
&\dd\lim_{n\to\infty}u_n(x,t)=u(x,t)\;\;\mbox{locally uniformly in}\;\;\R^N\times [0,\infty);\\[1mm]
&\dd\lim_{n\to\infty}U_n(x,t)=U(x,t)\;\;\mbox{locally uniformly in}\;\;\R^N\times [0,\w].
\eess
These facts, together with the limit \qq{5.6}, allow us to derive that \qq{5.3} holds.
\end{proof}

\end{document}